\documentclass[11pt]{article}

\usepackage[T1]{fontenc}
\usepackage[utf8]{inputenc}
\usepackage{lmodern}
\usepackage{microtype}
\usepackage[a4paper,margin=1.1in]{geometry}
\usepackage{amsmath,amssymb,amsthm}
\usepackage{hyperref}

\hypersetup{
  colorlinks=true,
  linkcolor=blue,
  citecolor=blue,
  urlcolor=blue
}

\newtheorem{theorem}{Theorem}

\title{On the multiplicative group of a two-sided skew brace of solvable type}
\author{Marco Damele%
\thanks{Dipartimento di Matematica, Università di Cagliari, Via Ospedale 72, 09124 Cagliari, Italy; \texttt{marco.damele@unica.it}; ORCID 0009-0008-3088-5766}%
\thanks{The author was supported by INdAM and GNSAGA -- Gruppo Nazionale per le Strutture Algebriche, Geometriche e le loro Applicazioni, and by ProBiki, funded by Fondazione di Sardegna.}}
\date{}

\begin{document}

\maketitle

\begin{abstract}
We prove that if $(B,+,\cdot)$ is a two-sided skew brace whose additive group is solvable, then every finite quotient of the multiplicative group $(B,\cdot)$ is solvable. In particular, our result recovers Nasybullov's theorem in the finite case ~\cite[Theorem~4.3(1)]{Nas} and extends it to arbitrary two-sided skew braces of solvable type.
\end{abstract}

{\footnotesize\noindent\textbf{2020 Mathematics Subject Classification.} Primary 16T25; Secondary 20F16, 16N20.\par}

{\footnotesize\noindent\textbf{Keywords.} Skew brace, two-sided skew brace, solvable group, finite quotient, radical ring.\par}

\bigskip

Skew braces, introduced by Guarnieri and Vendramin~\cite{GV}, provide an algebraic framework for the study of set-theoretic solutions of the Yang--Baxter equation and are closely related to Hopf--Galois structures. We also refer to Damele and Loi~\cite{DL} for recent structural results in the Lie setting.
Recall that a \emph{skew brace} is a triple $(B,+,\cdot)$ such that $(B,+)$ and $(B,\cdot)$ are groups and
\[
a\cdot (b+c)=a\cdot b-a+a\cdot c
\]
for all $a,b,c\in B$, where $-a$ denotes the inverse of $a$ in the additive group $(B,+)$. A skew brace $(B,+,\cdot)$ is said to be \emph{two-sided} if it also satisfies
\[
(a+b)\cdot c=a\cdot c-c+b\cdot c
\]
for all $a,b,c\in B$. We say that $(B,+,\cdot)$ is \emph{of solvable type} if the additive group $(B,+)$ is solvable.
A basic open problem in the area, originating in Byott's work on Hopf--Galois structures~\cite{Byott} and explicitly recorded by Smoktunowicz and Vendramin~\cite[Question~1.25]{SV}, asks whether every finite skew brace of solvable type has solvable multiplicative group. This question remains open in general. Some positive results are known. If $(B,+)$ is nilpotent, then $(B,\cdot)$ is solvable by ~\cite[Corollary~1.23]{SV}. Gorshkov and Nasybullov proved the conjecture when $|B|$ is not divisible by $3$ and showed that a minimal counterexample cannot have simple multiplicative group~\cite{GN}. Stefanello and Trappeniers established the corresponding statement for bi-skew braces~\cite{ST}. For two-sided skew braces the situation is more rigid. Nasybullov proved that every finite two-sided skew brace of solvable type has solvable multiplicative group~\cite[Theorem~4.3(1)]{Nas}. On the other hand, he also constructed infinite two-sided skew braces with abelian additive group and non-solvable multiplicative group~\cite[Section~3]{Nas}, showing that the naive infinite analogue fails. Nevertheless, his examples suggest a weaker phenomenon: although the multiplicative group is not solvable, all of its finite quotients are solvable. The following theorem shows that this is indeed the case.

\begin{theorem}\label{thm:main}
Let $(B,+,\cdot)$ be a two-sided skew brace such that $(B,+)$ is solvable. Then, for every normal subgroup $N\trianglelefteq (B,\cdot)$ of finite index, the quotient $(B,\cdot)/N$ is solvable.
\end{theorem}

\begin{proof}
Let \(n\) denote the derived length of the additive group \((B,+)\). We argue by induction on \(n\).
Assume first that $n=1$. Then $(B,+)$ is abelian, so $(B,+,\cdot)$ is a two-sided brace in the classical sense. By~\cite[Theorem~8.1.20]{CV}, there exists a Jacobson radical ring structure on the additive group $(B,+)$ whose adjoint group coincides with $(B,\cdot)$. By~\cite{ADS}, the adjoint group of a radical ring is an SN-group, that is, it admits a subnormal series with abelian factors. Quotients of SN-groups are again SN-groups, and every finite SN-group is solvable. Hence every finite quotient of $(B,\cdot)$ is solvable.
Assume now that $n>1$, and let
\[
D=(B,+)^{(1)}.
\]
Since $D$ is characteristic in $(B,+)$, Nasybullov's lemma on characteristic subgroups of the additive group of a two-sided skew brace implies that $D$ is an ideal of $B$~\cite[Lemma~4.1]{Nas}. In particular, $D$ is itself a two-sided skew brace, and $B/D$ is again a two-sided skew brace.
Let $N\trianglelefteq (B,\cdot)$ be of finite index and set
\[
F=(B,\cdot)/N.
\]
Let
\[
\pi\colon (B,\cdot)\twoheadrightarrow F
\]
be the natural epimorphism. Since $D$ is an ideal, $(D,\cdot)$ is a normal subgroup of $(B,\cdot)$, and therefore
\[
\pi(D)\cong (D,\cdot)/(D\cap N).
\]
Because $F$ is finite, the subgroup $\pi(D)$ is finite. Moreover, the additive group $(D,+)$ has derived length strictly smaller than $n$. By the induction hypothesis, every finite quotient of $(D,\cdot)$ is solvable. In particular, $\pi(D)$ is solvable.
Since $\pi(D)$ is a normal subgroup of $F$, the quotient $F/\pi(D)$ is defined. By the third isomorphism theorem,
\[
F/\pi(D)\cong (B,\cdot)/(DN).
\]
On the other hand, because $D$ is an ideal, we have a natural isomorphism
\[
(B/D,\cdot)\cong (B,\cdot)/D.
\]
Hence
\[
(B,\cdot)/(DN)
\cong \bigl((B,\cdot)/D\bigr)\big/\bigl((DN)/D\bigr)
\cong (B/D,\cdot)\big/\bigl((DN)/D\bigr).
\]
Thus $F/\pi(D)$ is a finite quotient of $(B/D,\cdot)$.
Now
\[
(B/D,+)\cong (B,+)/D
\]
is abelian, because $D=(B,+)^{(1)}$. Therefore the base case applies to the two-sided skew brace $B/D$, and every finite quotient of $(B/D,\cdot)$ is solvable. In particular, $F/\pi(D)$ is solvable.
We have shown that $\pi(D)$ is a solvable normal subgroup of $F$ and that the quotient $F/\pi(D)$ is solvable. Therefore $F$ is solvable. This completes the induction and the proof.
\end{proof}

\end{document}